\documentclass[12pt,a4paper]{amsart}
\usepackage[latin1]{inputenc}
\usepackage{graphicx}
\usepackage{amssymb, amsmath, mathrsfs}
\usepackage{geometry}
\usepackage{enumerate}
\usepackage[colorlinks=true,linkcolor=blue,urlcolor=blue,citecolor=blue]{hyperref}

\input xy
\xyoption{all}

\newtheorem{theorem}{Theorem}

\newtheorem{lemma}[theorem]{Lemma}
\newtheorem{corollary}[theorem]{Corollary}
\newtheorem{proposition}[theorem]{Proposition}

\theoremstyle{definition}
\newtheorem{example}[theorem]{Example}
\newtheorem{remark}[theorem]{Remark}

\par
\begin{document}

\title[Interpolation of strictly singular operators]{Interpolation and extrapolation of strictly singular operators between $L_p$ spaces}

\author[F. L. Hern\'{a}ndez]{Francisco L. Hern\'{a}ndez}
\address{F. L. Hern\'{a}ndez\\ Departamento de An\'{a}lisis Matem\'{a}tico, Universidad Complutense de Madrid, 28040, Madrid, Spain.}
\email{pacoh@mat.ucm.es}

\author[E. M. Semenov]{Evgeny M. Semenov}
\address{E. M. Semenov\\ Department of Mathematics, Voronezh State University, Voronezh 394006 (Russia).}
\email{nadezhka\_ssm@geophys.vsu.ru}

\author[P. Tradacete]{Pedro Tradacete}
\address{P. Tradacete\\ Departamento de Matem\'aticas\\ Universidad Carlos III de Madrid\\ 28911, Legan\'es, Madrid, Spain.}
\email{ptradace@math.uc3m.es}

\thanks{Research partially supported by Spanish Government Grant MTM2016-76808-P and Grupo UCM 910346. Second author was supported by Russian grant RFBR 17-01-00138. Third author also partially supported by MTM2016-75196-P}



\maketitle

\begin{abstract}
We study the interpolation and extrapolation properties of strictly singular operators between different $L_p$ spaces. To this end, the structure of strictly singular non-compact operators between \,$L_p-L_q$ spaces is analyzed. Among other things, we clarify the relation between strict singularity and the $L$-characteristic set of an operator. In particular, Krasnoselskii's interpolation theorem for compact operators is extended to the class of strictly singular operators.
\end{abstract}

\tableofcontents

\section{Introduction}

The classical Riesz-Thorin interpolation theorem for operators between $L_p$ spaces has a well-known counterpart, due to M. A. Krasnoselskii \cite{Krasnoselskii}, concerning the compactness properties of the operators involved (see \cite[Theorem IV.2.9]{BS}, also \cite{KZPS}). Given the range of applicability of Krasnoselskii's result, an analogue for strictly singular operators would be desirable. The  aim of the present paper is to provide the version of this principle for strictly singular operators, and to extend several results in this line that have been given for endomorphisms on $L_p$ spaces in \cite{HST}.

Recall that an operator between Banach spaces is strictly singular if it is never invertible on a (closed) infinite dimensional subspace. This class of operators forms a closed  two-sided operator ideal, containing that of compact operators, and was introduced by T. Kato \cite{Kato} in connection with the perturbation theory of Fredholm operators. Certain properties of compact operators, such as the spectral theory, can be extended to the class of strictly singular operators (see for instance \cite{G}), whereas many other facts are not in general true for this larger class; these include in particular the lack of non-trivial invariant subspaces \cite{Read} and interpolation properties \cite{Heinrich, Beucher}.

Despite the lack of interpolation properties in the general case, we will show that strict singularity can be interpolated for operators between different $L_p$ spaces extending Krasnoselskii's interpolation result. It should be noted that the classical proof of Krasnoselskii's theorem is based on the approximation of a compact operator by finite-rank ones, while the infinite dimensional nature of strictly singular operators requires a  different approach. In   Theorem \ref{t:interpol} we show that if  an operator $T:L_{p_0}\rightarrow L_{q_0}$ is strictly singular and $T:L_{p_1}\rightarrow L_{q_1}$ is bounded, then $T:L_{p_\theta}\rightarrow L_{q_\theta}$ is strictly singular for each $(p_\theta,q_\theta)$ in the interior of the interpolation segment joining $(p_0,q_0)$ with $(p_1,q_1)$ (that is, for $\frac1{p_\theta}=\frac{1-\theta}{p_0}+\frac{\theta}{p_1}$ and $\frac1{q_\theta}=\frac{1-\theta}{q_0}+\frac{\theta}{q_1}$, with $\theta\in(0,1)$). The case of endomorphisms, that is when $p_0=q_0$ and $p_1=q_1$, was considered in \cite{HST}. In particular, \cite[Theorem 4.2]{HST} asserts that for $p_0,p_1\in[1,\infty]$ if $T:L_{p_0}\rightarrow L_{p_0}$ is bounded and $T:L_{p_1}\rightarrow L_{p_1}$ is strictly singular, then $T:L_{p_\theta}\rightarrow L_{p_\theta}$ is actually compact for every $\theta\in(0,1)$. This interpolation result has been recently applied to the study of the rigidity of composition operators on Hardy spaces $H^p$  (see \cite{LNST}).

On the other hand, another extrapolation property of strict singularity  is given in Theorem \ref{t:extrapol}: Under certain assumptions on the indices $p_0,p_1,q_0,q_1\in(1,\infty)$, if an operator $T:L_{p_i}\rightarrow L_{q_i}$ is  bounded  for $i=0,1$, and for some $0 <\theta <1 $, $T:L_{p_{\theta}}\rightarrow L_{q_{\theta}}$  is strictly singular, then we actually have that $T:L_{p_\tau} \rightarrow L_{q_\tau}$  is compact for every $\tau\in(0,1)$.

Recall that given an operator $T:L_\infty\rightarrow L_1$, its $L$-{\it characteristic } is the set $L(T)$ of those $(\frac1p,\frac1q)\in[0,1]\times[0,1]$ such that the operator $T:L_p\rightarrow L_q$ is bounded. This classical notion, formally introduced by M. A. Krasnoselskii and P. Zabreiko in \cite{ZK}, provides relevant information about the interpolation properties of the operator $T$. In particular, Riesz-Thorin  convexity theorem yields that the $L$-characteristic is always a convex set. The subsets of the unit square arising as the $L$-characteristic set of an operator were characterized in \cite{R}. Our contribution in this respect is Theorem \ref{t:boundary}, which allows us to identify points $(\frac1p,\frac1q)$ for which an operator $T:L_p\rightarrow L_q$ is strictly singular but not compact, as part of the boundary of the $L$-characteristic set of $T$.

The paper is organized as follows: After recalling some preliminaries on the subspaces of $L_p$ and operators in Section \ref{s:preliminar}, we will analyze the structure of the sets $V_{p,q}$ of strictly singular non-compact operators between $L_p-L_q$ spaces in Section \ref{s:vpq}. In this section, we also study the stability under duality of strict singularity and the relation with the $L$-characteristic set $L(T)$ of an operator $T$ showing the inclusion  \, $V(T) \subset  \partial L(T)$. Section \ref{s:extrapolation} is devoted to prove a useful extrapolation property of strictly singular operators between $L_p$ spaces. Our results yield, under certain conditions, which will later be shown to be necessary, that a bounded operator mapping $T:L_{p_0}\rightarrow L_{q_0}$ and $T:L_{p_1}\rightarrow L_{q_1}$, which is strictly singular as an operator $T:L_p\rightarrow L_q$ for some pair $(p,q)$ in the interior of the interpolation segment joining $(p_0,q_0)$ and $(p_1,q_1)$, then it is necessarily compact for every point in the whole segment. In Section \ref{s:interpolation}, we give the analogue of Krasnoselskii compact interpolation result for strictly singular operators (see  Theorem \ref{t:interpol}). Finally, Section \ref{s:sss} collects some facts about strictly singular operators from the space $L_\infty$ and also extends some previous results to more general Banach lattices (see Theorem \ref{t:ell2compact}).
\bigskip

\section{Preliminaries}\label{s:preliminar}

Throughout, for $1\leq p\leq \infty$, $L_p$ denotes the space $L_p[0,1]$ equipped with Lebesgue measure $\mu$. Given Banach spaces $X$, $Y$, let $L(X,Y)$ denote the space of bounded linear operators from $X$ to $Y$. Also let $K(X,Y)$ and $S(X,Y)$ denote the corresponding ideals of compact and strictly singular singular operators, respectively.

Recall that an order continuous Banach lattice with weak unit is order isometric to a dense ideal in $L_1(\Omega,\Sigma,\mu)$ for certain probability space (cf. \cite[Theorem 1.b.14]{LT2}).
In the understanding of subspaces of $L_p$, or more generally, order continuous Banach lattices, Kade\v{c}-Pe\l czy\'nski dichotomy provides a useful tool (\cite{KP}, see also \cite[Proposition 1.c.8]{LT2}):

\begin{theorem}\label{Kadec-Pelc}
Let $X$ be a separable subset of an order continuous Banach lattice $E$. The following alternative holds:
\begin{enumerate}
    \item Either $X$ is a strongly embedded subset, that is, the norms of $E$ and $L_1$ are equivalent restricted to $X$; or
    \item $X$ contains an almost disjoint normalized sequence, that is, there exists a normalized sequence $(x_n)\subset X$ such that $x_n=u_n+v_n$, where $(u_n)$ is a disjoint sequence and $\|v_n\|\rightarrow 0$.
    \end{enumerate}
\end{theorem}

In particular, for $p>2$ this result yields that every normalized weakly null sequence in $L_p$ has a subsequence which is equivalent either to the unit basis of $\ell_p$ or $\ell_2$ (cf. \cite[Corollary 5]{KP}.

Next result, due to L. Dor \cite{Dor} (cf. \cite[Theorem 44]{AO}), also provides relevant information about the subspaces of $L_p$ which are isomorphic to $\ell_p$:

\begin{theorem}\label{Dor}
Let $1\leq p\neq 2<\infty$, $0<\theta\leq1$, and a sequence $(f_i)_{i=1}^\infty$ in $L_p$. Assume that either:
\begin{enumerate}
\item {$1\leq p<2,\, \|f_i\|\leq1$ for $i\in\mathbb N$, and for every $n\in \mathbb{N}$ and scalars $(a_i)_{i=1}^n$, we have  $$\|\sum_{i=1}^n a_if_i\|_p\geq\theta(\sum_{i=1}^n|a_i|^p)^{1/p},$$}
\item {or $2<p<\infty,\, \|f_i\|\geq1$ for $i\in\mathbb N$, and for every $n\in \mathbb{N}$ and scalars $(a_i)_{i=1}^n$, we have $$\|\sum_{i=1}^n a_if_i\|_p\leq\theta^{-1}(\sum_{i=1}^n|a_i|^p)^{1/p}.$$}
\end{enumerate}
Then there exist disjoint measurable sets $(A_i)_{i=1}^\infty$ in $[0,1]$ such that for every $i\in\mathbb N$ $$\|f_i\,\chi_{A_i}\|_p\geq\theta^{2/|p-2|}.$$
\end{theorem}

Recall the following interpolation property of compact operators between $L_p$ spaces due to M. A. Krasnoselskii (\cite{Krasnoselskii}, see also \cite[Theorem IV.2.9]{BS}, \cite{KZPS}):

\begin{theorem}\label{Krasnoselskii}
Let $1\leq p_0,p_1,q_0,q_1\leq \infty$. If $T:L_{p_0}\rightarrow L_{q_0}$ is a compact operator and $T:L_{p_1}\rightarrow L_{q_1}$ is bounded, then $T:L_{p_\theta}\rightarrow L_{q_\theta}$ is compact, where $\frac1{p_\theta}=\frac{1-\theta}{p_0}+\frac{\theta}{p_1}$ and $\frac1{q_\theta}=\frac{1-\theta}{q_0}+\frac{\theta}{q_1}$, for every $\theta\in(0,1)$.
\end{theorem}

Note that an analogous result for interpolating strictly singular operators does not hold in general. Indeed, one can consider the formal inclusion $T:L_\infty\rightarrow L_1$, which is strictly singular by a classical result of Grothendieck (cf. \cite[Theorem 5.2]{Ru}, see also Section \ref{s:sss}) and $T:L_1\rightarrow L_1$ is bounded. However, for $1<p<\infty$, $T:L_p\rightarrow L_1$ is invertible on the span of the Rademacher functions, thus it is not strictly singular. Apparently, positive results for one-sided interpolation of strictly singular operators are only known in the degenerated case when the initial couple reduces to one single space (see \cite[Proposition 2.1]{Beucher}, \cite{CMMM}, \cite[Proposition 1.6]{Heinrich}). Nevertheless, the following was given in \cite[Theorem 4.2]{HST} for endomorphisms on $L_p$ spaces:

\begin{theorem}\label{interpolationPAMS}
Let $1\leq r,s\leq \infty$, $r\neq s$ and $T:L_s\rightarrow L_s$ be a bounded operator. If $T\in S(L_r,L_r)$, then $T\in K(L_p,L_p)$ for every $p$ between $r$ and $s$.
\end{theorem}

Recall that an operator $T:E\rightarrow X$ defined from a Banach lattice $E$ to a Banach space $X$ is called \emph{AM-compact} when $T[-f,f]$ is relatively compact for every $f\in E_+$, where $[-f,f]=\{x\in E:|x|\leq f\}$. Also, $T:E\rightarrow X$ is called \emph{M-weakly compact} if $\|Tx_n\|\rightarrow 0$ for every disjoint norm-bounded sequence $(x_n)\subset E$.

Given an operator $T:L_\infty\rightarrow L_1$, its $L$-{\it characteristic} (cf. \cite{KZPS, ZK}) is the set
$$
L(T)=\Big\{\Big(\frac1p,\frac1q\Big)\in[0,1]\times[0,1]:T\in L(L_{p},L_{q})\Big\}.
$$
It is clear that if $(\alpha_1,\beta_1)\in L(T)$, then $(\alpha,\beta)\in L(T)$ whenever $0\leq \alpha \leq\alpha_1$ and $\beta_1\leq \beta\leq 1$. Riesz-Thorin  interpolation theorem asserts that $L(T)$ is a convex subset of $[0,1]\times[0,1]$. Actually, the subsets of $[0,1]\times[0,1]$ which can arise as the $L$-characteristic set of some operator are exactly the $F_\sigma$-sets with the above properties (see \cite{R}).

We follow standard notation concerning Banach spaces and Banach lattices. For any unexplained terminology the reader is referred to the monographs \cite{AK,LT1,LT2,MN}.
\bigskip

\section{Strictly singular non-compact operators}\label{s:vpq}

Given $1\leq p,q\leq \infty$, let us consider the set
$$
V_{p,q}=S(L_p,L_q) \backslash  K(L_p,L_q).
$$
The following result was announced in \cite{STH}:

\begin{theorem}\label{t:dokl}
Let $p,q\in[1,\infty)$.Then $V_{p,q}=\emptyset$ if and only if $p \geq 2 \geq q$.
\end{theorem}

\begin{proof}
If $p<2$ or $q>2$, we can consider respectively the operators $T$ and $S$ defined by
$$
\xymatrix{L_p\ar_{P_p}[d]\ar^T[rr]&&L_q&&L_p\ar_{P_{rad}}[d]\ar^S[rr]&&L_q\\
\ell_p\ar@{^{(}->}[rr]&&\ell_2 \ar_{J_{rad}}[u]&&\ell_2\ar@{^{(}->}[rr]&&\ell_q \ar_{J_q}[u] }
$$
where $P_p$, $P_{rad}$ are projections onto the closed linear span of disjointly supported functions in $L_p$ and respectively the closed linear span of the Rademacher functions, and $J_{rad}$, $J_q$ the embeddings via the Rademacher functions and a sequence of normalized disjointly supported functions in $L_q$. The operators $T$ and $S$ are respectively strictly singular and non-compact.

Let $p\geq 2\geq q$ and suppose $T\notin K(L_p,L_q)$. Then there exist a sequence $(x_i)_{i\in \mathbb N}$ in $L_p$ and $\lambda>0$ such that $\|x_i\|_p=1$, $x_i\overset{w}{\rightarrow} 0$ and $\|Tx_i\|_q\geq\lambda$ for every $i\in\mathbb N$. By \cite{KP}, there is a subsequence $(x_{i_k})$ which is equivalent to the unit basis of $\ell_p$ or $\ell_2$.

If $(x_{i_k})$ were equivalent to the unit basis of $\ell_p$, then since $L_q$ has cotype 2, there exist constants $c_1,c_2>0$ such that for every $n\in\mathbb N$ we have
$$
c_1 n^{\frac12}\leq\int_0^1\Big\|\sum_{k=1}^n r_k(t) Tx_{i_k}\Big\|_q dt\leq\|T\|\int_0^1\Big\|\sum_{k=1}^n r_k(t) x_{i_k}\Big\|_p dt\leq c_2 n^{\frac1p}.
$$
Since $p\geq2$, this is a contradiction for large enough $n$. Hence, the sequence $(x_{i_k})$ must be equivalent to the unit basis of $\ell_2$. Now, by \cite[Proposition 2.1]{FHKT} it follows that $T$ is not strictly singular.
\end{proof}

Recall that an operator $T$ between Banach spaces is compact if and only if its adjoint $T^*$ is compact. In general, this is no longer true for strictly singular operators. However, for endomorphisms on $L_p$ spaces this duality holds (see \cite{M} and \cite{Weis:77}). We will address now the case when the operator is defined between different $L_p$ spaces. Recall that a Banach space $X$ is called subprojective when every subspace contains a further subspace which is complemented in $X$. The following is well-known:

\begin{proposition}\label{p:pleq2}
Let $p\in(1,2]$ and $X$ a Banach space. If $T:L_p\rightarrow X$ is strictly singular, then $T^*$ is strictly singular.
\end{proposition}

\begin{proof}
By \cite{KP}, the space $L_{p}^*$ is subprojective, hence the result follows from \cite[Corollary 2.3]{W}.
\end{proof}

\begin{theorem}\label{t:dualitySS}
Let $p,q\in(1,\infty)$. There is a strictly singular operator $T:L_p\rightarrow L_q$ such that $T^*$ is not strictly singular, if and only if $p>q>2$.
\end{theorem}

\begin{proof}
If $p>q>2$, then $p'<q'<2$ so we can consider an isomorphic embedding $j_1:\ell_{q'}\rightarrow L_{p'}$ (using for instance $q'$-stable random variables). Let $j_2:\ell_q\rightarrow L_q$ be an isomorphic embedding generated by a sequence of disjointly supported functions, which span a complemented subspace, and let $T:L_p\rightarrow L_q$ be the operator given by $T=j_2 j_1^*$. Since $p>q>2$, by \cite[Corollary 5]{KP}, every subspace of $L_p$ contains a subspace isomorphic to $\ell_p$ or $\ell_2$, hence the operator $j_1^*:L_p\rightarrow \ell_q$ is strictly singular, and so is $T$. However, $T^*$ is an isomorphism on a subspace isomorphic to $\ell_{q'}$ hence it is not strictly singular.

Conversely, suppose now $p\leq q$ or $q\leq 2$. The case when $p\leq 2$ is contained in Proposition \ref{p:pleq2}. If $p\geq 2\geq q$, then by Theorem \ref{t:dokl}, every strictly singular operator $T:L_p\rightarrow L_q$ is compact, hence $T^*$ is also compact and in particular strictly singular. As mentioned above, the case $p=q$ has been given in \cite{M} and \cite{Weis:77}. Finally, if $q>p>2$, and $T:L_p\rightarrow L_q$ is strictly singular, then using \cite{J} we get the existence of $T_1:L_p\rightarrow \ell_q$ and $T_2:\ell_q\rightarrow L_q$ such that $T=T_2T_1$. Hence, we have $T^*=T_1^*T_2^*$. Note that $q'<p'<2$, so $L_{p'}$ does not contain any subspace isomorphic to $\ell_{q'}$ (cf. \cite[Proposition 6.4.3.]{AK}), thus  $T_1^*:\ell_{q'}\rightarrow L_{p'}$ is strictly singular and so is $T^*$.
\end{proof}

Recall that two measurable functions $f$ and $g$ are equi-measurable if for every $-\infty<\lambda<\infty$ the distribution functions satisfy
$$
\mu(\{s:f(s)> \lambda\})=\mu(\{s:g(s)>\lambda\}).
$$
In particular, given a measurable function $f$, we denote by $f^*$ its decreasing rearrangement:
$$
f^*(t)=\inf\Big\{\lambda>0:\mu\{s:|f(s)|>\lambda\}\leq t\Big\}.
$$

\begin{proposition}\label{lemma-29-2}
If $2<q\leq p<\infty$ and $T\in V_{p,q}$, then there exists a normalized sequence $(y_k)$ in $L_p$, which is equivalent to the unit vector basis of $\ell_2$ with $(|y_k|)$ equi-measurable, $(Ty_k)$ is equivalent to the unit vector basis of $\ell_q$, and the subspace $[y_k]$ is complemented.
\end{proposition}

\begin{proof}
We proceed as in \cite[Lemma 3.2]{HST} where the case $p=q$ was done. Assume $q<p$, since $T\notin K(L_p,L_q)$ then there exists a sequence $(x_k)$ in
$L_p$, such that $\|x_k\|_p=1$,
$x_k\mathop{\to}\limits^{w}0$ and $\|Tx_k\|_q\geq\alpha$
for some $\alpha>0$. By \cite[Corollary 5]{KP}, we may suppose that $(x_k)$ (respectively, $(Tx_k)$) is
equivalent to the unit vector basis of $\ell_r$ with $r\in\{2,p\}$ (resp., $\ell_s$ with $s\in\{2,q\}$).

The cases (i) $r=s=2$, (ii) $r=p,\,s=2$, and (iii) $r=p,\,s=q$ cannot happen. Indeed, in case (i), the restriction of $T$ on the subspace
$[x_k]$ is an isomorphism, which contradicts the assumption that $T\in S(L_p,L_q)$. While, if case (ii) holds, then for every $n\in\mathbb N$ we have
\[
n^{\frac12}\approx  \left\|\sum_{k=1}^nTx_k\right\|_q\leq \|T\|\left\|\sum_{k=1}^n x_k\right\|_p\approx n^{\frac1p},
\]
which is impossible for large $n\in\mathbb N$ as $p>2$. Similarly, in case (iii), we would have
\[
n^{\frac1q}\approx  \left\|\sum_{k=1}^nTx_k\right\|_q\leq \|T\|\left\|\sum_{k=1}^n x_k\right\|_p\approx n^{\frac1p},
\]
which again is impossible for large $n\in\mathbb N$ as $p>q$. Hence, $(x_k)$ is necessarily equivalent to the unit vector basis of $\ell_2$ and
$(Tx_k)$ is equivalent to the unit vector basis of $\ell_q$.

Now, by \cite[Theorem 3.2]{SS} there is a subsequence $(x_{n_k})$ such that $x_{n_k}=u_k+v_k+w_k$, where
\begin{enumerate}
    \item $\|u_k\|_p\leq1$, $(|u_k|)$ are equi-measurable, and $u_k\mathop{\to}\limits^w0$;
    \item $|v_i|\wedge |v_j|=0$ for any $i\ne j$ in $\mathbb{N}$, with $\|v_k\|_p\leq2$, and $v_k\mathop{\to}\limits^w0$;
    \item $\lim\limits_{k\to\infty}\|w_k\|_p=0$.
\end{enumerate}

It holds that $\lim\limits_{k\to\infty}\|Tv_k\|_q=0$. Indeed, otherwise we can select a subsequence $(v_{i_k})$ such that $\inf\limits_k
\|Tv_{i_k}\|_q>0$. By  \cite[Corollary 5]{KP} some subsequence of $(Tv_{i_k})$ is equivalent to the unit vector basis of $\ell_2$
or $\ell_q$. As above, both cases are impossible because $(v_{i_k})$ is equivalent to the unit vector basis of $\ell_p$ and $p>q>2$.

Now, since
$\lim\limits_{k\to\infty}\|w_k\|_p=0$ we have that $\lim\limits_{k\to\infty}\|Tw_k\|_q=0$, and so
\[
 \lim_{k\to\infty}\|Tx_{n_k}-Tu_k\|_q\leq\lim_{k\to\infty}\left(\|Tv_k\|_q+\|Tw_k\|_q\right)=0.
\]
Thus, by a standard perturbation (cf. \cite[Thm. 1.a.9]{LT1}), it follows that $(Tu_{j_k})$ is also equivalent to the unit vector basis of $\ell_q$. Since $u_k\mathop{\to}\limits^w0$, we have that $(u_{j_k})$ must be equivalent to the unit vector basis of $\ell_2$ and spans a complemented subspace of $L_p$. Set $y_k=u_{j_k}$, which satisfies the required properties.
\end{proof}

Following a similar approach as with the $L$-characteristic of an operator, given $T:L_\infty\rightarrow L_1$, let us consider the \emph{$V$-characteristic} set
$$
V(T)=\Big\{\Big(\frac1p,\frac1q\Big)\in(0,1)\times(0,1): T\in V_{p,q}\Big\}.
$$
The following result relates the set $V(T)$ with the boundary of $L(T)$, and extends the preliminary result given in \cite[Theorem 3.7]{HST}) for endomorphisms. 

\begin{theorem}\label{t:boundary}
For every operator $T:L_\infty\rightarrow L_1$ the following inclusion holds:
$$
V(T)\subseteq \partial L(T).
$$
\end{theorem}

Before the proof we need the following.

\begin{lemma}\label{l:qdisjoint}
Let $1\leq p,q<\infty$ with $q\neq2$. If an operator $T:L_p\rightarrow L_q$ and $(x_n)\subset L_p$ are such that $\sup_n\|x_n\|_p<\infty$ and $(Tx_n)$ is equivalent to the unit vector basis of $\ell_q$, then $T\notin L(L_p,L_r)$ for any $r>q$.
\end{lemma}

\begin{proof}
By Theorem \ref{Dor}, there exist $\lambda>0$ and a sequence of disjoint subsets $A_n\subset[0,1]$ such that $\|(Tx_n) \chi_{A_n}\|_q\geq\lambda$ for every $n\in\mathbb N$. Given $r>q$, let $s>q$ be such that $\frac1r+\frac1s=\frac1q$. By Holder's inequality it follows that
$$
\lambda\leq\|(Tx_n)\chi_{A_n}\|_q\leq\|Tx_n\|_r\|\chi_{A_n}\|_s=\|Tx_n\|_r\mu(A_n)^{\frac1s}.
$$
Since $\mu(A_n)\rightarrow0$ and $s<\infty$ we have that $\|Tx_n\|_r\rightarrow\infty$. Thus, $T\notin L(L_p,L_r)$, as claimed.
\end{proof}

\begin{proof}[Proof of Theorem \ref{t:boundary}]
Let $T\in V_{p,q}$, and suppose $(\frac1p,\frac1q)\notin \partial L(T)$. Let us assume first that $q>2$. We claim that $T$ is AM-compact: To see this, first let  $M\in\mathbb R_+$, and $(f_n)\subset L_p$ be such that $|f_n|\leq M$. We want to prove that some subsequence $(Tf_{n_k})$ converges in norm. If this were not the case, without loss of generality, we can assume that some subsequence $(f_n)$ is normalized weakly null and for some $\lambda>0$ we have $\|Tf_n\|_q>\lambda$ for every $n\in\mathbb N$. By \cite[Lemma 1.4]{FHKT}, there is a subsequence such that $(f_{n_k})$ is equivalent to the unit vector basis of $\ell_2$. Since $T$ is strictly singular, $(Tf_{n_k})$ has no subsequence which is equivalent to the unit vector basis of $\ell_2$. Therefore, by \cite[Corollary 5]{KP}, as $q>2$ it follows that there is a subsequence of $(Tf_{n_k})$ which is equivalent to the unit vector basis of $\ell_q$. Lemma \ref{l:qdisjoint} shows then that $(\frac1p,\frac1q)\in \partial L(T)$. This is a contradiction, hence, we can assume that $T[-M,M]$ is a relatively compact set in $L_q$ for every $M\geq0$.

Now, for arbitrary $f\in L_p$, and any $\varepsilon>0$, taking $M_\varepsilon\in\mathbb R_+$ such that $\|(|f|-M_\varepsilon)_+\|_p\leq\varepsilon$ we have that
$$
[-|f|,|f|]\subset[-M_\varepsilon,M_\varepsilon]+\varepsilon B_{L_p}.
$$
Since $T[-M_\varepsilon,M_\varepsilon]$ is relatively compact, it follows that $T[-|f|,|f|]$ is also relatively compact in $L_q$. Thus, $T$ is AM-compact, as claimed.

Since $T$ is not compact, there is a normalized weakly null sequence $(x_n)$ in $L_p$ such that $c=\inf_n\|Tx_n\|_q>0$. By \cite{KP}, there is a subsequence of $(x_n)$, not relabelled, that can be writen as $x_n=y_n+z_n$ where $(y_n)$ are pairwise disjoint and $(z_n)$ are a $p$-equi-integrable sequence, that is for every $\varepsilon>0$ there is $f_\varepsilon\in L_p$ such that
$$
(z_n)\subset [-f_\varepsilon,f_\varepsilon]+\varepsilon B_{L_p}.
$$

Note that since  $(x_n)$ and $(y_n)$ are weakly null, then so is $(z_n)$. Therefore, since
$$
(T(z_n))\subset T[-f_\varepsilon,f_\varepsilon]+\|T\|\varepsilon B_{L_q},
$$
by the AM-compactness of $T$, it follows that $\liminf_n\|Tz_n\|_q=0$.

Hence, we have found a sequence $(y_n)\subset L_p$ of pairwise disjoint seminormalized elements, such that $\|Ty_n\|_q\geq c>0$ for every $n\in\mathbb N$. Let $A_n\subset [0,1]$ be such that $y_n=y_n\chi_{A_n}$ (which clearly satisfy $\mu(A_n)\rightarrow0$). Suppose now that $T:L_s\rightarrow L_q$ were bounded for some $s<p$. Hence, taking $r>s$ such that $\frac1p+\frac1r=\frac1s$, by Holder's inequality we have
$$
c\leq\|Ty_n\|_q\leq \|T\|\|y_n\|_s=\|T\|\|y_n\chi_{A_n}\|_s\leq\|T\|\|y_n\|_p\|\chi_{A_n}\|_r=\|T\|\|y_n\|_p\mu(A_n)^{\frac1r}.
$$
Since $r<\infty$ and $\|y_n\|_p\leq1$, this is a contradiction. Therefore, $(\frac1p,\frac1q)\in \partial L(T)$ as claimed. This proves the statement when $q>2$.

When $p<2$, we have by Proposition \ref{p:pleq2}, together with Schauder's theorem, that if $T\in V_{p,q}$, then $T^*\in V_{q',p'}$, with $p'>2$. Hence, by the previous part of the proof it follows that $(\frac{1}{q'},\frac{1}{p'})\in\partial L(T^*)$, which is tantamount to $(\frac1p,\frac1q)\in\partial L(T)$.

The only remaining case would be that $q\leq 2\leq p$, but since in this case $V_{p,q}=\emptyset$ by Theorem \ref{t:dokl}, the statement is trivially true.
\end{proof}

\begin{remark}
Note that the inclusion $V(T)\subset \partial L(T)\cap(0,1)\times(0,1)$ can be strict. For instance, the formal inclusion $T:L_\infty\rightarrow L_1$ given by $Tf=f$ is easily seen to satisfy
$V(T)=\emptyset$,
whereas
$$
L(T)=\Big\{\Big(\frac1p,\frac1q\Big):1\leq q\leq p\leq \infty\Big\}.
$$
Analogous examples with $V(T)\neq\emptyset$ can also be constructed using the examples in \cite{R}.
\end{remark}

For regular operators (i.e., those which can be written as a difference of two positive operators) strict singularity is even closer to compactness as the following shows (this result was given in \cite{CG}, but we include a proof here for convenience).

\begin{theorem}\label{t:positive}
Let $T:L_p\rightarrow L_q$ be a regular operator with $1<q\leq p<\infty$. If $T$ is strictly singular, then it is compact.
\end{theorem}

\begin{proof}
The case when $p=q$ is given in \cite[Corollary 3.5]{HST}. Hence, let us suppose $q<p$. We claim that in this case every regular operator $T:L_p\rightarrow L_q$ is M-weakly compact. Indeed, let $(x_n)$ be a norm bounded disjoint sequence in $L_p$. As $T$ is regular, its modulus $|T|$ defines a bounded operator. Since $(|x_n|)$ is weakly null, it follows that $(|T|(|x_n|))$ is also weakly null. Moreover, as $|T|(|x_n|)\geq0$ for every $n\in\mathbb N$, we have that $\||T|(|x_n|)\|_1\rightarrow0$. If $\||T|(|x_n|)\|_q\rightarrow 0$, then we are done, so let us assume the contrary. By Kadec-Pelczynski Theorem \ref{Kadec-Pelc}, we can assume that the sequence $(|T|(|x_n|))$ is almost disjoint in $L_q$. Hence, for every $n\in\mathbb N$ we have
$$
n^{\frac1q}\approx\Big\|\sum_{k=1}^n |T|(|x_k|)\Big\|_q\leq\|T\|\Big\|\sum_{k=1}^n |x_k|\Big\|_p\approx n^{\frac1p},
$$
which is a contradiction with $q<p$.

Note that if $q\leq 2\leq p$, then by Theorem \ref{t:dokl} we know that strictly singular operators are compact. If $q<p<2$, then we have that $T:L_r\rightarrow L_q$ is strictly singular for every $r\geq p$, so by Theorem \ref{t:dokl} we have that $T:L_r\rightarrow L_q$ is compact for $r\geq2$, so Krasnoselskii Theorem \ref{Krasnoselskii} implies that $T:L_r\rightarrow L_q$ is actually compact for every $r> p$. From this, it is easy to see that $T:L_p\rightarrow L_q$ is AM-compact. Hence, by \cite[Proposition 3.7.4]{MN} we get that $T:L_p\rightarrow L_q$ is compact.

Finally, it remains the case when $2<q<p$. Now, $T:L_p\rightarrow L_s$ is strictly singular for $s\leq q$, so again by Theorem \ref{t:dokl} we have that $T:L_p\rightarrow L_s$ is compact for $s\leq 2$, and Krasnoselskii Theorem \ref{Krasnoselskii} yields that $T:L_p\rightarrow L_s$ is compact for $s<q$. By Schauder's theorem we have that $T^*:L_{s'}\rightarrow L_{p'}$ is compact for $s<q$, where $\frac1s+\frac1{s'}=1=\frac1p+\frac1{p'}$, and as before we get that $T^*:L_{q'}\rightarrow L_{p'}$ is AM-compact. Hence, since $T^*$ is $M$-weakly compact, by \cite[Proposition 3.7.4]{MN} we get that $T^*:L_{q'}\rightarrow L_{p'}$ is compact. Therefore, $T:L_p\rightarrow L_q$ is compact.
\end{proof}

We will later see in Theorem \ref{t:paralel} that the conditions in Theorem \ref{t:positive} cannot be relaxed, i.e., there exist regular operators $T\in V_{p,q}$ for $p<q$. Similarly for $p=q=1$.

\bigskip

\section{Extrapolation}\label{s:extrapolation}

The following is an extrapolation property of strict singularity and extends a previous result given for the case of endomorphisms in \cite[Theorem 3.3]{HST}. Throughout, let  $1\leq p_0,p_1,q_0,q_1 \leq \infty$, and for each  $\theta\in(0,1)$, let $p_\theta$ and $q_\theta$ be given by
$$
\frac{1}{p_\theta}=\frac{\theta}{p_0}+\frac{1-\theta}{p_1}\,\,\textrm{ and }\,\,\frac{1}{q_\theta}=\frac{\theta}{q_0}+\frac{1-\theta}{q_1}.
$$

\begin{theorem}\label{t:extrapol}
Let $1<p_i,q_i<\infty$ for $i=0,1$ with  $q_0\neq q_1$ , $p_0\neq p_1$. Suppose either
\begin{itemize}
\item $\min\{\frac{q_0}{p_0},\frac{q_1}{p_1}\}\leq 1$, or
\item $\min\{\frac{q_0}{p_0},\frac{q_1}{p_1}\}> 1$ and $\frac{q_1-q_0}{p_1-p_0}<0$.
\end{itemize}
If $T$ is a bounded operator from $L_{p_i}$ to $L_{q_i}$ for $i=0,1$, and for some $0 < \theta <1 $, $T\in S(L_{p_{\theta}},L_{q_{\theta}})$, then $T\in K(L_{p_\tau},L_{q_\tau})$ for every $\tau\in(0,1)$.
\end{theorem}

\begin{proof}
We will split the proof in three cases:
\begin{enumerate}
\item[(a)] When $\max\{\frac{q_0}{p_0},\frac{q_1}{p_1}\}\leq 1$.
\item[(b)] When $\min\{\frac{q_0}{p_0},\frac{q_1}{p_1}\}\leq 1<\max\{\frac{q_0}{p_0},\frac{q_1}{p_1}\}$.
\item[(c)] When $\min\{\frac{q_0}{p_0},\frac{q_1}{p_1}\}> 1$ and $\frac{q_1-q_0}{p_1-p_0}<0$.
\end{enumerate}
\medskip

(a) Supppose that $\max\{\frac{q_0}{p_0},\frac{q_1}{p_1}\}\leq 1$. Without loss of generality, let us assume $q_0<q_\theta<q_1$. Let us suppose that $T\notin K(L_{p_\tau},L_{q_\tau})$ for some $\tau\in(0,1)$. By Krasnoselskii's Theorem \ref{Krasnoselskii}, we must then have that for the given $\theta\in(0,1)$, $T\in V_{p_{\theta},q_{\theta}}$. Suppose first that $2< q_{\theta}\leq p_{\theta}$, and let $(y_k)\subset L_{p_{\theta}}$ be the sequence obtained by Proposition \ref{lemma-29-2}. Since the sequence $(Ty_k)$ is equivalent to the unit vector basis of $\ell_{q_{\theta}}$, by Dor's Theorem \ref{Dor}, there exist $c>0$ and disjoint sets $(A_k)\subset [0,1]$, such that for every $k\in\mathbb N$
\begin{equation}\label{eq:dor}
\|(Ty_k)\chi_{A_k}\|_{q_{\theta}}\geq c.
\end{equation}

Since $(y_k)$ are equi-measurable, by Hardy-Littlewood inequality (cf. \cite[\S 2, Lemma 2.1]{BS}) we have for every $k\in \mathbb N$
$$
\int_A|y_{k}|^{p_\theta} \leq \int _{0}^{\mu(A) } (y_{1}^*)^{p_\theta}.
$$
Where $y_1^*$ denotes the decreasing rearrangement of $y_1$, and hence of every $y_k$. Therefore, there is $\varepsilon>0$ such that whenever $\mu(A)<\varepsilon$, it follows that
\begin{equation}\label{eq:equi-integrable}
\|y_k\chi_A\|_{p_\theta}<\frac{c}{2\|T\|_{p_\theta,q_\theta}}
\end{equation}
for every $k\in\mathbb N$ (in other words, $(y_k)$ are uniformly $p_\theta$-integrable).

The equi-measurability of $(y_k)$ also implies the existence of measurable subsets $C_k\subset[0,1]$ with $\mu(C_k)\geq1-\varepsilon$, such that $y_k\chi_{C_k}\in L_\infty$ with $\|y_k\chi_{C_k}\|_\infty\leq y_1^*(\varepsilon)$ for every $k\in\mathbb N$. Let
$$
r_\theta=\frac{q_0q_1}{\theta(q_1-q_0)}.
$$
It follows that $\frac{1}{q_\theta}=\frac{1}{q_1}+\frac{1}{r_\theta}$. Now, using H\"{o}lder's inequality and the fact that
$$
\|y_k\chi_{C_k}\|_{p_1}\leq\|y_k\chi_{C_k}\|_\infty\leq y_1^*(\varepsilon)
$$
we have
\begin{align}
    \|(T(y_k\chi_{C_k}))\chi_{A_k}\|_{q_\theta} &\leq\|T(y_k\chi_{C_k})\|_{q_1}\|\chi_{A_k}\|_{r_\theta}\label{eq:mu(A_k)}\\
\nonumber    &\leq\|T\|_{p_1,q_1}\,y^*_1(\varepsilon)\,\mu(A_k)^{\theta\left(\frac1{q_0}-\frac1{q_1}\right)}.
\end{align}
On the other hand, since $\mu([0,1]\backslash C_k)<\varepsilon$, by \eqref{eq:equi-integrable} we have
\begin{equation}\label{eq:c/2}
    \|T(y_k\chi_{[0,1]\backslash C_k})\|_{q_\theta}\leq \|T\|_{p_\theta,q_\theta}\|y_k\chi_{[0,1]\backslash C_k}\|_{p_\theta}<\frac{c}{2}.
\end{equation}

Therefore, we have
\begin{align*}
\|(Ty_k)\chi_{A_k}\|_{q_\theta}&\leq\|(T(y_k\chi_{C_k}))\chi_{A_k}\|_{q_\theta}+\|T(y_k\chi_{[0,1]\backslash C_k})\|_{q_\theta}\\
&\leq\|T\|_{p_1,q_1}\,y^*_1(\varepsilon)\,\mu(A_k)^{\theta\left(\frac1{q_0}-\frac1{q_1}\right)}+\frac{c}{2}.
\end{align*}
and since $\mu(A_k)\rightarrow 0$ this is a contradiction with \eqref{eq:dor} for large $k\in\mathbb N$.

This finishes the proof in the case that $2< q_{\theta}\leq p_{\theta}$. The case when $q_{\theta}\leq p_{\theta}<2$ reduces to the previous one by a standard duality argument based on Theorem \ref{t:dualitySS}. Finally, Theorem \ref{t:dokl} yields that $V_{p_\theta,q_\theta}=\emptyset$ when $q_{\theta}\leq 2\leq p_{\theta}$ and the claim follows.
\medskip

(b) Supppose that $\min\{\frac{q_0}{p_0},\frac{q_1}{p_1}\}\leq 1<\max\{\frac{q_0}{p_0},\frac{q_1}{p_1}\}$. Without loss of generality, we can assume that $q_0\leq p_0$, $p_1<q_1$ and $p_\theta\leq q_\theta$. Let
$$
m=\frac{\frac1{p_0}-\frac1{p_1}}{\frac1{q_0}-\frac1{q_1}}.
$$
Assume first that $q_0<p_0$. We distinguish two cases:

Suppose first that $m< 0$. Since $T\in S(L_{p_{\theta}},L_{q_{\theta}})$, we have that $T\in S(L_p,L_q)$ whenever $p\geq p_\theta$ and $q\leq q_\theta$. Since $m<0$, we can take $\theta'\in(0,1)$ such that $p_{\theta'}>p_\theta$, $q_{\theta'}<q_\theta$ and $q_{\theta'}<p_{\theta'}$. Therefore, $T\in S(L_{p_{\theta'}},L_{q_{\theta'}})$ and by part (a) we must have $T\in K(L_{p_{\theta'}},L_{q_{\theta'}})$. By Krasnoselskii Theorem \ref{Krasnoselskii}, the conclussion follows.

Now, suppose that $m>0$. Pick as before $\theta'\in(0,1)$ such that $p_{\theta'}>p_\theta$, $q_{\theta'}<q_\theta$ and $q_{\theta'}<p_{\theta'}$. Since $T\in S(L_{p_{\theta}},L_{q_{\theta}})$, part (a) together with Theorem \ref{Krasnoselskii} imply that
\begin{equation}\label{q<q_theta'}
T\in K(L_{p_{\theta'}},L_{q})\,\,\,\textrm{whenever}\,q<q_{\theta'},
\end{equation}
and
\begin{equation}\label{p>p_theta'}
T\in K(L_{p},L_{q_{\theta'}})\,\,\,\textrm{whenever}\,p>p_{\theta'}.
\end{equation}

From \eqref{p>p_theta'}, it easily follows that $T:L_{p_{\theta'}}\rightarrow L_{q_{\theta'}}$ is AM-compact. We will see that $T:L_{p_{\theta'}}\rightarrow L_{q_{\theta'}}$ is also $M$-weakly compact. Indeed, otherwise there exist a disjoint normalized sequence $(x_n)\subset L_{p_{\theta'}}$ such that $\|Tx_n\|_{q_{\theta'}}\geq\alpha>0$ for every $n\in\mathbb N$. By Kadec-Pelczynski dichotomy, either $\|Tx_n\|_{q_{\theta'}}\approx \|Tx_n\|_1$, or $(Tx_n)$ has an almost disjoint subsequence. The former case is impossible because of \eqref{q<q_theta'}. Hence, passing to a subsequence we can assume that $(Tx_n)$ is equivalent to the unit basis of $\ell_{q_{\theta'}}$. In particular, for every $N\in\mathbb N$ we have
$$
N^{\frac{1}{q_{\theta'}}}\approx \Big\|\sum_{n=1}^N T x_n\Big\|_{q_{\theta'}}\leq\|T\|\Big\|\sum_{n=1}^N x_n\Big\|_{p	_{\theta'}}\approx N^{\frac1{p_{\theta'}}},
$$
which is a contradiction with the fact that $q_{\theta'}<p_{\theta'}$. Therefore, it follows that $T:L_{p_{\theta'}}\rightarrow L_{q_{\theta'}}$ is M-weakly compact, and by \cite[Proposition 3.7.4]{MN} it follows that $T\in K(L_{p_{\theta'}},L_{q_{\theta'}})$. Again, by Krasnoselskii Theorem \ref{Krasnoselskii}, the conclusion follows.

It remains to consider the case when $q_0=p_0$  with $p_{1}< q_{1}$. Since $T\in S(L_{p_{\theta}},L_{q_{\theta}})$ for some $\theta\in(0,1)$, we have that $T\in S(L_p,L_q)$ for $p\geq p_\theta$ and $q\leq q_\theta$. By part (a), it follows that $T\in K(L_p,L_q)$ for $p\geq p_\theta$ and $q\leq q_\theta$. If $m<0$ this already implies that $T\in K(L_{p_0},L_{q_0})$ and by Krasnoselskii's Theorem \ref{Krasnoselskii} the conclusion follows. Finally, if $m>0$, then arguing as above with statements similar to \eqref{q<q_theta'} and \eqref{p>p_theta'}, it follows that $T\in K(L_{p_{\theta}},L_{q_{\theta}})$ for the given $\theta\in(0,1)$, and again by Krasnoselskii Theorem \ref{Krasnoselskii} the proof is finished.
\medskip

(c) Suppose $\min\{\frac{q_0}{p_0},\frac{q_1}{p_1}\}> 1$ and $\frac{q_1-q_0}{p_1-p_0}<0$. Without loss of generality assume that $p_1<p_0$. Note that we clearly have $T\in L(L_{p_1},L_{p_1})$ and $T\in L(L_{q_1},L_{q_1})$. Since $T\in S(L_{p_\theta},L_{q_\theta})$ we have $p_1<q_\theta<q_1$ with $T\in S(L_{q_\theta},L_{q_\theta})$. Thus, by part (a) it follows that $T\in K(L_r,L_r)$ for every $p_1<r<q_1$. Now, by Krasnoselskii Theorem \ref{Krasnoselskii} we get that $T\in K(L_{p_\tau},L_{q_\tau})$ for every $\tau\in(0,1)$.
\end{proof}

\begin{remark}
The above result does not hold when $p_0=p_1$ or $q_0=q_1$. Indeed, consider the operators
$$
\xymatrix{L_p\ar_{P_p}[d]\ar^T[rr]&&L_q&&L_p\ar_{P_{rad}}[d]\ar^R[rr]&&L_q\\
\ell_p\ar@{^{(}->}[rr]&&\ell_2 \ar_{J_{rad}}[u]&&\ell_2\ar@{^{(}->}[rr]&&\ell_q \ar_{J_q}[u] }
$$
with $p<2$ in the first case and $2<q$ in the second, $P_p$, $P_{rad}$ are projections onto the span of disjointly supported functions in $L_p$ and respectively the span of the Rademacher functions, and $J_{rad}$, $J_q$ the embeddings via the Rademacher functions and a sequence of normalized disjointly supported functions in $L_q$.

Now, note that for $1 \leq p_0\leq p_\theta\leq p_1\leq 2$ and $1<q<\infty$, we have that $T\in V_{p_\theta,q}$. Similarly, for $2\leq q_0 \leq q_\theta \leq q_1<\infty$ and $1<p<\infty$, the above operator satisfies $R\in V_{p,q_\theta}$.
\end{remark}

\begin{remark}
Theorem \ref{t:extrapol} does not hold when $p_0=p_1\in\{1,\infty\}$. Indeed, let $T:L_\infty\rightarrow L_q$ be the formal inclusion operator. We have that $T\in V_{\infty,q}$ for every $q\in[1,\infty)$. Analogously, given a sequence of paiwise disjoint sets $(A_n)_{n=1}^\infty\subset[0,1]$, each of them having positive measure, let $T:L_1\rightarrow L_q$ be given by
$$
Tf=\sum_{n=1}^\infty \Big(\int_{A_n}f d\mu\Big)\,r_n,
$$
where $r_n$ denotes the $n$-th Rademacher function. It follows that $T\in V_{1,q}$ for every $q\in(1,\infty)$.
\end{remark}

\begin{remark}
Notice that an alternative proof of Theorem \ref{t:boundary} can be given  using the above extrapolation result. Indeed, assume that there exists $(\frac1p,\frac1q) \in V(T)\backslash\partial L(T)$. Clearly, $(\frac1p,\frac1q)$ must belong to the interior of $L(T)$. Now we can deduce from the extrapolation Theorem \ref{t:extrapol}, by taking a suitable interpolation line segment in each case, that $T \in K(L_{p},L_{q})$. This is a contradiction.
\end{remark}

\begin{corollary}
Let $T:L_\infty\rightarrow L_1$ and assume $V(T)$ contains a line segment $I$ with $(\frac1{p_0},\frac1{q_0})\in I$ for some $1<q_0\leq p_0$. It follows that either
\begin{enumerate}
\item $p_0<2$ and $I$ is a vertical segment, or
\item $q_0>2$ and $I$ is a horizontal segment.
\end{enumerate}
\end{corollary}

\begin{proof}
Since $T\in V_{p_0,q_0}$ with $1<q_0\leq p_0$, by Theorem \ref{t:dokl}, it follows that $p_0<2$ or $q_0>2$. Now, take another point $(\frac1{p_1},\frac1{q_1})\in I$ and assume $p_1\neq p_0$ and $q_1\neq q_0$. Note that $(\frac1{p_\theta},\frac1{q_\theta})\in I\subset V(T)$ for every $\theta\in(0,1)$. In particular, $T\in S(L_{p_\theta},L_{q_\theta})$ for every $\theta\in (0,1)$. Since $\min\{\frac{q_0}{p_0},\frac{q_1}{p_1}\}\leq 1$, by Theorem \ref{t:extrapol}, we get that $T\in K(L_{p_\theta},L_{q_\theta})$, which is in contradiction with the fact that $I\subset V(T)$. Therefore, we must have that $p_1=p_0<2$ or $q_1=q_0>2$ and $I$ is respectively a vertical or horizontal segment.
\end{proof}

In general, Theorem \ref{t:extrapol} cannot be extended to the situation when $\min\{\frac{q_0}{p_0},\frac{q_1}{p_1}\}>1$ as the following shows.

\begin{theorem}\label{t:paralel}
For each $\lambda\in(0,1)$ set $S_\lambda=\{(p,q):1\leq p,q\leq \infty,\,\frac1p-\frac1q=\lambda\}$. Then
$$
\bigcap_{(p,q)\in S_\lambda}V_{p,q}\neq\emptyset.
$$
\end{theorem}

\begin{proof}
Given $\lambda\in(0,1)$ , let $(A_k)_{k\in\mathbb N}$ be a sequence of pairwise disjoint measurable sets in $[0,1]$ with $\mu(A_k)>0$. For $f\in L_1$, let us consider the operator defined by
$$
Tf=\sum_{k\in\mathbb N} \left(\mu(A_k)^{\lambda-1}\int_{A_k} fd\mu \right)\chi_{A_k}.
$$
We claim that for any $1\leq p<q\leq\infty$ such that $\frac1p-\frac1q=\lambda$, $T$ defines a bounded linear operator $T:L_p\rightarrow L_q$ with $\|T\|=1$. Indeed, note first that using H\"older's inequality it follows that
$$
\int_{A_k} fd\mu \leq\| f\chi_{A_k}\|_{p}\|\chi_{A_k}\|_{p'}=\| f\chi_{A_k}\|_{p}\,\mu(A_k)^{1-\frac1p}.
$$
Hence,
\begin{align*}
\|Tf\|_{q}&\leq \Big(\sum_{k\in\mathbb N}\Big(\mu(A_k)^{\lambda-1}\| f\chi_{A_k}\|_{p}\,\mu(A_k)^{1-\frac1p}	\Big)^q\mu(A_k)\Big)^{\frac1q}	 \\
&=\Big(\sum_{k\in\mathbb N}\| f\chi_{A_k}\|_{p}^q\Big)^{\frac1q}	\\
&\leq\Big(\sum_{k\in\mathbb N}\| f\chi_{A_k}\|_{p}^p\Big)^{\frac1p}	\\
&\leq\|f\|_{p}.
\end{align*}
This shows that $\|T\|\leq1$, and since for any $k\in\mathbb N$
\begin{equation}\label{A_k}
T(\mu(A_k)^{-\frac1p}\chi_{A_k})=\mu(A_k)^{-\frac1q}\chi_{A_k},
\end{equation}
it follows that $\|T\|=1$, as claimed.

Now, observe that $T:L_p\rightarrow L_q$ can be written as
$$
\xymatrix{L_p\ar_{P}[d]\ar^T[rr]&&L_q\\
\ell_p\ar@{^{(}->}^i[rr]&&\ell_q \ar_{Q}[u]}
$$
where $P$ is the averaging projection onto the linear span of $(\chi_{A_k})$ in $L_p$, $i=i_{p,q}$ is the formal inclusion from $\ell_p$ to $\ell_q$, and $Q$ is the isomorphic embedding of $\ell_q$ in $L_q$ via the functions $(\chi_{A_k})$. Therefore, $T\in S(L_p,L_q)$ and taking \eqref{A_k} into account it follows that $T\notin K(L_p,L_q)$. Since this holds for every $1\leq p<q\leq \infty$ such that $\frac1p-\frac1q=\lambda$ we have that
$$
T\in \bigcap_{(p,q)\in S_\lambda}V_{p,q}.
$$
\end{proof}

\medskip

\begin{example}\label{example2parts}
\textit{Given $1 < p< 2 < q <\infty$, there is an operator  $T:L_\infty\rightarrow L_1$ such that
$$
V(T) = \Big\{\Big(\frac1p,\frac1q\Big) \Big\}
$$
}

Indeed, let us consider a function  $g \in L_q \setminus \cup_{r >q} \, L_r [0, \frac12]$ \, and  \,$ h \in L_{p'}  \setminus \cup_{r >p' } L_r [0, \frac12]$,  where \,$\frac{1}{p}+\frac{1}{p' }= 1$.  Define the rank-one operator $T_1:L_p[0,\frac12] \rightarrow L_{q}[0,\frac12]$ given by
$$
T_{1} f(t) = \left(\int_{0}^{ \frac12} f(s) h(s) d\mu \right) \,  g(t).
$$
Due to the choice of $g$ and $h$, it is easy to see that the $L$-characteristic set of $T_1$ is
$$
L(T_1) = \Big\{\Big(\frac1r,\frac1s\Big): p\leq r \leq \infty , 1\leq s \leq q  \Big\},
$$
and since $T\in K(L_r,L_s)$ for $p\leq r \leq \infty,\, 1\leq s \leq q$, we have $V(T_{1}) = \emptyset$.

Now, let $T_{2}: L_p[\frac12, 1]\rightarrow L_{q}[\frac12, 1]$ be the operator given by
$$
T_2f=\sum_{k\in\mathbb N} \left(\mu(A_k)^{\lambda-1}\int_{A_k} fd\mu\right) \chi_{A_k} ,
$$
where $(A_k)_{k\in\mathbb N}$ is a sequence of pairwise disjoint measurable sets in $[\frac12,1]$ with $\mu(A_k)>0$ and $\lambda=\frac1{p}-\frac1{q}$. It follows that
$$
L(T_{2})=  \Big\{\Big(\frac1r,\frac1s\Big):  \frac1r- \frac1s\leq \lambda  \Big\}
$$
while
$$
V(T_{2})=\Big\{\Big(\frac1r,\frac1s\Big):  \frac1r - \frac1s = \lambda  \Big\}.
$$
Finally, consider the operator $T(f) = T_{1}(f\chi_{[0,\frac12]}) + T_{2}(f\chi_{[\frac12,1]})$, which satisfies that $L(T) = L(T_{1})$ and $V(T) = \{(\frac1p,\frac1q)\}$.
\end{example}

\vspace{2mm}
\begin{example}\label{example3parts}
\textit{Given $p_1< 2<q_0$ and $p_0<q_0$, $p_1<q_1$ such that $\frac{1}{q_0}-\frac{1}{p_0}=\frac{1}{q_1}-\frac{1}{p_1}=\lambda$, there exists an operator $T:L_\infty\rightarrow L_1$ such that
$$
\begin{array}{ccl}
V(T)&=&\{(\frac1p,\frac1q):\frac{1}{q}-\frac{1}{p}=\lambda, \,p\in[p_0,p_1],\,q\in[q_0,q_1]\}\,\cup\\
&&\\
&&\{(\frac1p,\frac1{q_0}):p\geq p_0\}\,\cup\,\{(\frac1{p_1},\frac1q):q\leq q_1\}
\end{array}
$$
}

Indeed, let $T_1:L_p[0,\frac13]\rightarrow L_{q_0}[0,\frac13]$ be given by the composition of projecting onto the span of the Rademacher sequence on $[0,\frac13]$, then the formal inclusion $i_{2,q_0}:\ell_2\rightarrow \ell_{q_0}$ and finally embedding $\ell_{q_0}$ into $L_{q_0}[0,\frac13]$ via a sequence of disjoint functions; let $T_2:L_p[\frac13,\frac23]\rightarrow L_{q}[\frac13,\frac23]$ be given by
$$
T_2f=\sum_{k\in\mathbb N} \left(\mu(A_k)^{\lambda-1}\int_{A_k} fd\mu\right) \chi_{A_k},
$$
where $(A_k)_{k\in\mathbb N}$ is a sequence of pairwise disjoint measurable sets in $[\frac13,\frac23]$ with $\mu(A_k)>0$ and $\lambda=\frac1{p_0}-\frac1{q_0}$; let $T_3:L_{p_1}[\frac23,1]\rightarrow L_{q}[\frac23,1]$ be given by projecting first on the span of a sequence of pairwise disjoint functions in $L_{p_1}[\frac23,1]$, then compose with the formal inclusion $i_{p_1,2}:\ell_{p_1}\rightarrow \ell_2$, and finally compose this with the embedding of $\ell_2$ in $L_q[\frac23,1]$ via a sequence of Rademacher functions. Now the operator $T=T_1+T_2+T_3$  has the desired properties.
\end{example}

\begin{figure}
\centering
\ifx\JPicScale\undefined\def\JPicScale{0.5}\fi
\unitlength \JPicScale mm
\begin{picture}(110,110)(0,0)
\linethickness{0.3mm}
\put(-0.2,0){\line(1,0){100.1}}
\linethickness{0.3mm}
\put(0,0){\line(0,1){100}}
\linethickness{0.3mm}
\put(80,0){\line(0,1){80}}
\put(0,80){\line(1,0){80.2}}
\linethickness{0.3mm}
\put(0,20){\line(1,0){35}}
\linethickness{0.3mm}
\multiput(35,20)(0.12,0.12){210}{\line(1,0){0.12}}
\linethickness{0.3mm}
\put(60,45){\line(0,1){35}}
\put(35,15){\makebox(0,0)[cc]{\tiny{$(\!\frac1{p_0},\!\frac1{q_0}\!)$}}}
\put(71,45){\makebox(0,0)[cc]{\tiny{$\!(\!\frac1{p_1},\!\frac1{q_1}\!)$}}}
\put(98,0){\makebox(0,0)[cc]{$>$}}
\put(105,0){\makebox(0,0)[cc]{\tiny{$\frac1p$}}}
\put(0,98){\makebox(0,0)[cc]{$\wedge$}}
\put(0,107){\makebox(0,0)[cc]{\tiny{$\frac1q$}}}
\put(-2,40){\makebox(0,0)[cc]{\tiny{$\frac12-$}}}
\put(40,-3){\makebox(0,0)[cc]{\tiny{$\overset{|}{\frac12}$}}}
\end{picture}
\caption{The set $V(T)$ of Example \ref{example3parts}.\label{figureVT}}
\end{figure}
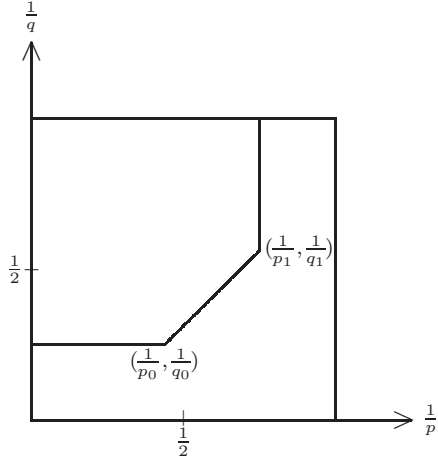

The above motivates the following open question: Given an operator $T:L_\infty\rightarrow L_1$, does the set $V(T)$ always look like that of Figure \ref{figureVT}?

\bigskip
\bigskip

\section{Interpolation}\label{s:interpolation}

Recall that, in general, strictly singular operators are not suitable for interpolation properties (cf. \cite{Beucher,Heinrich}). Here we analyze this question providing a positive answer for operators between $L_{p}-L_{q}$ spaces. Our previous result \cite[Theorem 4.2]{HST} could be considered as a preliminary version of this fact valid just for endomorphism on $L_{p}$-spaces.

\begin{theorem}\label{t:interpol}
Let $1\leq p_0,p_1,q_0,q_1<\infty$. If an operator $T:L_{p_0}\rightarrow L_{q_0}$ is strictly singular and $T:L_{p_1}\rightarrow L_{q_1}$ is bounded, then $T:L_{p_\theta}\rightarrow L_{q_\theta}$ is strictly singular for each $\theta\in(0,1)$.
\end{theorem}

\begin{proof}
Assume first that $q_\theta>1$. Note that $T\in S(L_{r},L_{s})$ whenever $(r,s)\in [p_0,\infty]\times[1,q_0]$. Therefore, by Theorem \ref{t:dokl} it follows that $T\in K(L_r,L_s)$ when $r\geq \max\{2,p_0\}$ and $s\leq\min\{q_0,2\}$. Hence, by Kranoselskii's Theorem \ref{Krasnoselskii} it follows that
\begin{equation}\label{eq:Krqtheta}
  T\in K(L_r,L_{q_\theta})\,\,\,\textrm{for every }r>p_\theta.
\end{equation}

Similarly, we also have that
\begin{equation}\label{eq:Kpthetas}
  T\in K(L_{p_\theta},L_s)\,\,\,\textrm{for every }s<q_\theta.
\end{equation}

Now, suppose, for the sake of contradiction that the operator $T:L_{p_\theta}\rightarrow L_{q_\theta}$ is not strictly singular. Thus, there exist an infinite dimensional subspace $X\subset L_{p_\theta}$ and $c>0$ such that
$$
\|Tx\|_{q_\theta}\geq c\|x\|_{p_\theta}
$$
for every $x\in X$, in other words, the restriction $T|_X$ is an isomorphism onto $T(X)$. Note that if the subspace $T(X)$ were strongly embedded in $L_{q_\theta}$, then for every $s<q_\theta$ we would have
$$
\|Tx\|_s\approx \|Tx\|_{q_\theta}\geq c\|x\|_{p_\theta}
$$
which is impossible by \eqref{eq:Kpthetas}. Therefore, by Kadec-Pelczynski Theorem \ref{Kadec-Pelc}, it follows that $T(X)$ contains an almost disjoint sequence, so going to a further subspace, we can assume that $X$ and $T(X)$ are isomorphic to $\ell_{q_\theta}$ and are complemented respectively in $L_{p_\theta}$ and $L_{q_\theta}$. Note that for $1<p<\infty$ the space $\ell_q$ is complemented in $L_p$ only when $q=2$ or $q=p$, while for $p=1$ this only holds when $q=1$ (see for instance \cite[Proposition 5.6.1 and Theorem 6.4.21]{AK}). Therefore, the statement follows whenever $p_\theta\neq q_\theta\neq 2$.

In order to complete the proof we will consider the following three cases separately:
\begin{enumerate}
\item $p_\theta=q_\theta$.
\item $p_\theta>2=q_\theta$.
\item $p_\theta<2=q_\theta$.
\end{enumerate}

Note that if $p_0=p_1$, then the result follows from (\cite{Beucher} Proposition 2.1 or \cite{Heinrich} Proposition 1.6); while if $q_0=q_1=q_\theta>1$, then in the first two cases, using the stability under duality Theorem \ref{t:dualitySS}, this can always be reduced to the previous one. Hence, in (1) and (2) , we can assume $p_0\neq p_1$ and $q_0\neq q_1$.

(1) In the case that $p_\theta=q_\theta$, we have several possibilities: either $p_0=q_0$ and $p_1=q_1$ in which case the result follows from Theorem \ref{interpolationPAMS}, or $\min\{\frac{q_0}{p_0},\frac{q_1}{p_1}\}<1$ in which case we can pick another $\theta'\in(0,1)\backslash\{\theta\}$ so that $p_{\theta'}\neq q_{\theta'}$, hence by the above part of the proof we have that $T\in S(L_{p_{\theta'}},L_{q_{\theta'}})$, but in this case, since $p_0\neq p_1$ and $q_0\neq q_1$, by Theorem \ref{t:extrapol}, it follows that $T\in K(L_{p_\theta},L_{q_\theta})$.
\medskip

(2) If $p_\theta>2=q_\theta$, then necessarily $\min\{\frac{q_0}{p_0},\frac{q_1}{p_1}\}<1$. Hence, if we pick $\theta'\in(0,1)\backslash\{\theta\}$ so that $p_{\theta'}\neq q_{\theta'}\neq 2$, then by the above part of the proof we have that $T\in S(L_{p_{\theta'}},L_{q_{\theta'}})$. As before, since $p_0\neq p_1$ and $q_0\neq q_1$, by Theorem \ref{t:extrapol}, it follows that $T\in K(L_{p_\theta},L_{q_\theta})$.
\medskip

(3) Assume now $p_\theta<2=q_\theta$. Note that since the sequence $(x_n)\subset X$ is  equivalent to the unit vector basis of $\ell_2$, then up to a further subsequence  $(x_n)$ must be equi-integrable in $L_{p_\theta}$. Indeed, by the subsequence splitting property, up to a further subsequence, we can write $x_n=g_n+h_n$ with $|g_n|\wedge|h_n|=0$, $(g_n)$ equi-integrable and $(h_n)$ disjoint. Now, let us suppose that the disjoint part $(h_{n})$ satisfy $\|h_n\|_{p_\theta}\geq K>0$ for every $n\in\mathbb N$. Then, by \cite[Theorem 1.d.6]{LT2} we have
$$
n^{\frac{1}{2}}\approx\Big\|\sum_{k=1}^n x_k\Big\|_{p_\theta}\approx\Big\|\Big(\sum_{k=1}^n |x_k|^2\Big)^{\frac12}\Big\|_{p_\theta}\geq\Big\|\Big(\sum_{k=1}^n |h_k|^2\Big)^{\frac12}\Big\|_{p_\theta}\approx n^{\frac{1}{p_\theta}},
$$
and since $p_\theta<2$ this is a contradiction for large $n$. Hence, we can assume that $(x_n)$ is equi-integrable in $L_{p_\theta}$. Now, it easily follows from \eqref{eq:Krqtheta} that $T:L_{p_\theta}\rightarrow L_2$ is in fact AM-compact, so $\|Tx_n\|_2\rightarrow 0$. This is a contradiction with the assumption that $T|_X$ was an isomorphism, and the proof is finished when $q_\theta>1$.

Finally, it remains to consider the case when $q_\theta=1$, which implies that $q_0=q_1=1$. In this case, since $T\in S(L_{p_0},L_1)$, we have that $T\in S(L_r,L_1)$ for every $r\geq p_0$, so in particular $T\in K(L_r,L_1)$ for $r\geq\max\{p_0,2\}$ by Theorem \ref{t:dokl}. Hence, by  Theorem \ref{Krasnoselskii}, it follows that $T\in K(L_{p_\theta},L_1)$.
\end{proof}

Note that the inclusion operator $T:L_\infty\rightarrow L_q$ is strictly singular for $1\leq q<\infty$, while $T:L_p\rightarrow L_q$ is invertible on the span of the Rademacher functions. This shows that $p_0<\infty$ in Theorem \ref{t:interpol} is a necessary condition.

\begin{corollary}\label{c:hit2<q<p}
Let $1<p_i,q_i<\infty$ for $i=0,1$ with  $q_0\neq q_1$, $p_0\neq p_1$. Suppose $\min\{\frac{q_0}{p_0},\frac{q_1}{p_1}\}\leq 1$, or $\min\{\frac{q_0}{p_0},\frac{q_1}{p_1}\}> 1$ and $\frac{q_1-q_0}{p_1-p_0}<0$. If an operator $T: L_{p_0}\rightarrow L_{q_0}$ is strictly singular and $T:L_{p_1}\rightarrow L_{q_1}$ is bounded, then $T\in K(L_{p_\theta},L_{q_\theta})$ for every $\theta\in(0,1)$.
\end{corollary}

\begin{proof}
This is a direct consequence of Theorems \ref{t:interpol} and  \ref{t:extrapol}.
\end{proof}

Note that the domain restriction in this Corollary is a necessary condition as Theorem \ref{t:paralel} shows.

Recall that an operator between Banach spaces $T:X\rightarrow Y$ is \emph{super strictly singular} (or finitely strictly singular) if for every $\varepsilon>0$ there is $N\in\mathbb N$ such that every subspace $F\subset X$ with dimension greater than $N$ contains a vector $x\in F$ such that $\|Tx\|\leq\varepsilon\|x\|$. This class of operators provide an asymptotic version of strictly singular operators (see \cite{Pli:04}).

In particular, an operator $T:X\rightarrow Y$ is super strictly singular if and only if, for every ultrafilter $\mathscr U$, the corresponding ultraoperator $T_{\mathscr U}:X_{\mathscr U}\rightarrow Y_{\mathscr U}$ is strictly singular. Since every ultrapower $(L_p)_{\mathscr U}$ is another $L_p$-space (over a larger measure space), it follows from Theorem \ref{t:interpol} that super strictly singular operators between $L_p$ spaces can also be interpolated.

Recall that for $T:X\rightarrow Y$ and every $n\in \mathbb N$, the Bernstein numbers $b_n(T)$ are defined as
$$
b_n(T)=\sup\Big\{\inf\Big\{\frac{\|Tx\|}{\|x\|}:x\neq 0\in F\Big\}: F\subset X,\,\textrm{dim}(F)=n\Big\}.
$$
These numbers play an important role in approximation theory and in the study of super strictly singular operators (see for instance \cite{FHR,HRS}). In particular, an operator $T$ is super strictly singular if and only if $b_n(T)\underset{n\rightarrow\infty}\longrightarrow0$. Considering the above comments, it would be interesting to study the interpolation properties of Bernstein numbers for operators between $L_p$ spaces.

\bigskip

\section{Strictly singular operators on other Banach lattices}\label{s:sss}

In this section we address the relation of strict singularity with compactness and related notions for operators between Banach lattices. Recall that given a Banach space $X$, an operator $T:E\rightarrow F$ is called $X$-\emph{singular} if it is never an isomorphism when restricted to a subspace isomorphic to $X$. The following rigidity result is an abstract version of Theorem \ref{t:dokl}, and can also be considered as a natural extension of \cite[Proposition 2.1]{FHKT}:

\begin{theorem}\label{t:ell2compact}
Let $E$ be a Banach lattice with type 2 and an unconditional basis, and $F$ be a Banach lattice satifying a lower 2-estimate. If an operator $T:E\rightarrow F$ is $\ell_2$-singular, then it is compact. In particular, $K(E,F)=S(E,F)$.
\end{theorem}

\begin{proof}
Let us assume that $T:E\rightarrow F$ is not compact. Since $E$ has type 2, it cannot contain subspaces isomorphic to $c_0$ nor $\ell_1$, hence by \cite[Theorem 1.c.5]{LT2}, the space $E$ is reflexive. Hence, there exists a normalized weakly null sequence $(x_n)\subset E$ such that for some $\delta>0$ we have $\|Tx_n\|_F\geq\delta$ for every $n\in\mathbb N$.

Without loss of generality, we can assume that $(x_n)$ is a 1-unconditional basic sequence. Hence, using that $E$ has type 2, there is $M>0$ such that for scalars
$(a_n)_{n=1}^m$ we have
$$
\Big\|\sum_{n=1}^m a_n x_n\Big\|_E=\int_0^1\Big\|\sum_{n=1}^m a_n r_n(t) x_n\Big\|_E dt\leq M\Big(\sum_{n=1}^m a_n^2\Big)^\frac12.
$$
Now, by Kadec-Pelczynski Theorem \ref{Kadec-Pelc}, either the sequence $(\|Tx_n\|_1)$ is
bounded away from zero, or $(Tx_n)$ has a subsequence equivalent to
a disjoint sequence.

Suppose first that $(\|Tx_n\|_{L_1})$ is bounded away from zero,
then by \cite[6.Theorem]{Aldous-Fremlin}, there exist $C>0$ and a subsequence $(Tx_{n_k})$ such that for scalars
$(a_k)_{k=1}^m$ we have
\begin{align*}
M\|T\|\Big(\sum_{k=1}^m|a_k|^2\Big)^{\frac{1}{2}} &\geq\|T\|\Big\|\sum_{k=1}^m a_k
x_{n_k}\Big\|_E \geq\Big\|\sum_{k=1}^m a_k
Tx_{n_k}\Big\|_F\\
&\geq\Big\|\sum_{k=1}^m a_k
Tx_{n_k}\Big\|_1\geq C\Big(\sum_{k=1}^m
|a_k|^2\Big)^{\frac{1}{2}}.
\end{align*}

On the other hand, if $(Tx_{n_k})$ is equivalent to a disjoint sequence, using that $F$ satisfies a lower 2 estimate, then for scalars
$(a_k)_{k=1}^m$ we would have
$$
M\|T\| \Big(\sum_{k=1}^m|a_k|^2\big)^{\frac{1}{2}}\geq \|T\|\Big\|\sum_{k=1}^m a_k
x_{n_k}\Big\|_E \geq\Big\|\sum_{k=1}^ma_kT(x_{n_k})\Big\|_F\geq K\delta\Big(\sum_{k=1}^m|a_k|^2\Big)^{\frac{1}{2}}.
$$
Thus, in both cases this would imply that $T$ is not $\ell_2$-singular, and we reach a contradiction.
\end{proof}

In particular, this can be applied when $E$ (respectively, F) is an Orlicz space $L^\varphi[0,1]$ with indices $2<\alpha_\varphi^\infty\leq\beta_\varphi^\infty<\infty$ (respectively, $1<\alpha_\varphi^\infty\leq\beta_\varphi^\infty<2$, see for instance \cite{LT2} for the definition of the indices).

\begin{remark}
The hypothesis that $E$ has an unconditional basis in Theorem \ref{t:ell2compact} can be weakened to having an unconditional finite dimensional decomposition, or even to the unconditional subsequence property.
\end{remark}

In the remaining of the section we address the case of strictly singular operators from $L_\infty$.

\begin{proposition}\label{p:wksss}
Let $X$ be a Banach space and an operator $T:L_\infty\rightarrow X$. The following statements are equivalent:
\begin{enumerate}
\item $T$ is super strictly singular.
\item $T$ is strictly singular.
\item $T$ is $c_0$-singular.
\item $T$ is weakly compact.
\end{enumerate}

\end{proposition}

\begin{proof}
The implications $(1)\Rightarrow (2)\Rightarrow(3)$ are trivial. Note that $L_\infty$ is isomorphic to a space $C(K)$ (take for instance $K$ to be $\beta\mathbb N$, the Stone-Cech compactification of $\mathbb N$). The implication $(3)\Rightarrow (4)$ was given in \cite{Pel}. The implication $(4)\Rightarrow (1)$ has been considered in a more general context in \cite{Le}, but we include a proof next.

If $T$ is weakly compact, by \cite[Theorem 15.2]{DJT} there is $p\in[1,\infty)$ and a probablity measure $\mu\in C(K)^*$, such that for every $\varepsilon>0$ there is $N(\varepsilon)>0$ such that
$$
\|Tf\|_X\leq N(\varepsilon)\|f\|_p+\varepsilon\|f\|_{\infty}.
$$
Since the inclusion $i:C(K)\hookrightarrow L_p(\mu)$ is strictly singular (cf. \cite[Theorem 5.2]{Ru}), letting $\varepsilon_1=\varepsilon/N(\varepsilon)$, there is some $n\in\mathbb N$ such that for every subspace $X\subseteq C(K)$ with $\mathrm{dim}(X)\geq n$, there is $x\in X$ such that $\|x\|_p\leq\varepsilon_1\|x\|_\infty$. Therefore, we have
$$
\|Tx\|_X\leq N(\varepsilon)\|x\|_p+\varepsilon\|x\|_{\infty}\leq 2\varepsilon\|x\|_\infty,
$$
which yields that $T:L_\infty\rightarrow X$ is super strictly singular.
\end{proof}

\begin{theorem}\label{t:sssLinf}
Let $X$ be a Banach space which is either separable or does not have any subspace isomorphic to $c_0$. Then every operator $T:L_\infty\rightarrow X$ is strictly singular.
\end{theorem}

\begin{proof}
Let $X$ be separable. Suppose $T:L_\infty \rightarrow X$ is not strictly singular. By Proposition \ref{p:wksss}, $T$ is an isomorphism on a subspace $E\subset L_\infty$ isomorphic to $c_0$. By Sobczyk's theorem (cf. \cite[Corollary 2.5.9.]{AK}) there is a projection $P:X\rightarrow T(E)$. Now, $T^{-1}|_{T(E)}PT$ defines a projection on $L_\infty$ onto $E$, which is isomorphic to $c_0$. This is impossible (cf. \cite[Theorem 2.5.5.]{AK}).

On the other hand if $X$ does not have any subspace isomorphic to $c_0$, the result follows directly from  Proposition \ref{p:wksss}.
\end{proof}

\begin{corollary}\label{c:noc0}
If $X$ be a rearrangement invariant space, then  $X$ is separable if and only if every operator $T:L_\infty\rightarrow X$ is strictly singular.
\end{corollary}

\begin{proof}
If $X$ is separable, the statement follows from Theorem \ref{t:sssLinf}. On the other hand, if $X$ is non-separable, then it contains a subspace isomorphic to $L_\infty$ and the corresponding embedding provides a non-strictly singular operator into $X$.
\end{proof}

The following is a quantified version of Theorem \ref{t:sssLinf} by means of estimating the corresponding Bernstein numbers.

\begin{theorem}\label{t:cotype}
Suppose $X$ is a Banach space with cotype $q\in[2,\infty)$, then there is a constant $K>0$ such that every operator $T\in L(L_\infty,X)$ is super strictly singular and for every $n\in\mathbb N$
$$
b_n(T)\leq K\|T\|n^{-1/q},
$$
\end{theorem}

\begin{proof}
Let $E$ be a $2n$-dimensional subspace of $L_\infty$. By  \cite[Lemma 1.3]{DJT}, there exist $x_1,\ldots, x_n\in E$ such that $1/2\leq\|x_k\|\leq1$ and
\begin{equation}\label{weak2sum}
\|\sum_{k=1}^n a_kx_k\|\leq\Big(\sum_{k=1}^na_k^2\Big)^{1/2}.
\end{equation}
This implies that
$$
\sup\Big\{\Big(\sum_{k=1}^n |x^*(x_k)|^2\Big)^{\frac12}:x^*\in L_\infty^*,\,\|x^*\|=1\Big\}\leq 1.
$$
Since $X$ has cotype $q\in[2,\infty)$, using \cite[Theorem 11.14]{DJT}, there is a constant $C>0$ such that
\begin{equation}
\Big(\sum_{k=1}^n \|Tx_k\|^q\Big)^{\frac1q}\leq C\|T\|\sup\Big\{\Big(\sum_{k=1}^n |x^*(x_k)|^2\Big)^{\frac12}:x^*\in L_\infty^*,\,\|x^*\|=1\Big\}\leq C\|T\|
\end{equation}
Therefore, there must be some $i_0\in\{1,\ldots,n\}$ such that $\|Tx_{i_0}\|\leq C\|T\|n^{-1/q}$, as claimed.
\end{proof}

\begin{corollary}
Let $1\leq q<\infty$. Every operator $T:L_\infty\rightarrow L_q$ is super strictly singular, and
$$
b_n(T)\leq K\|T\|n^{-\frac1{\max\{q,2\}}}.
$$
\end{corollary}


\begin{thebibliography}{99}

\bibitem{AK}
F. Albiac, N. J. Kalton,
Topics in Banach space theory.
Graduate Texts in Mathematics, 233. Springer, New York, 2006.

\bibitem{Aldous-Fremlin}
D. J. Aldous, D. H. Fremlin,
Colacunary sequences in $L$-spaces.
Studia Math. \textbf{71} (1981/82), no. 3, 297--304.

\bibitem{AO}
D. Alspach, E. Odell,
$L_p$ spaces.
Handbook of the geometry of Banach spaces, Vol. I, 123--159, North-Holland, 2001.

\bibitem{BS}
C. Bennett, R. Sharpley,
Interpolation of operators.
Pure and Applied Mathematics, 129. Academic Press, Inc., Boston, MA, 1988.

\bibitem{Beucher}
O. J. Beucher,
On interpolation of strictly (co-)singular linear operators.
Proc. Roy. Soc. Edinb. A \textbf{112} (1989), 263--269.

\bibitem{CG}
V.~Caselles, M.~Gonz\'{a}lez,
Compactness properties of strictly singular operators in Banach lattices.
Semesterbericht Funktionalanalysis. T\"{u}bingen, Sommersemester. (1987), 175--189.

\bibitem{CMMM}
F. Cobos, A. Manzano, A. Mart�nez, P. Matos. On interpolation of strictly singular operators, strictly co-singular operators and related operator ideals. Proc. Roy. Soc. Edinb. A 130 (2000), 971--989.

\bibitem{DJT}
J. Diestel, H. Jarchow, and A. Tonge,
Absolutely summing operators.
Cambridge University Press, 1995.

\bibitem{Dor}
L. Dor,
On projections in $L_1$.
Annals of Math. \textbf{102} (1975), 463--474


\bibitem{FHKT}
J. Flores, F. L. Hern\'andez, N. J. Kalton, and P. Tradacete,
Characterizations of strictly singular operators on Banach lattices.
J. London Math. Soc. (2) \textbf{79} (2009), no. 3, 612--630.

\bibitem{FHR}
J. Flores, F. L. Hern\'andez, and Y. Raynaud,
Super strictly singular and co-singular operators and related classes.
J. Operator Theory \textbf{67} (2012), 121--152.

\bibitem{G}
S. Goldberg,
Unbounded linear operators. Theory and applications. Reprint of the 1985 corrected edition.
Dover Publications, Inc., Mineola, NY, 2006.

\bibitem{Heinrich}
S. Heinrich,
Closed operator ideals and interpolation.
J. Funct. Analysis \textbf{35} (1980), 397--411.

\bibitem{HRS}
F. L. Hern\'andez, Y. Raynaud, and E. M. Semenov,
Bernstein widths and super strictly singular inclusions.
A panorama of modern operator theory and related topics, 359--376, Oper. Theory Adv. Appl., 218, Birkh\"auser/Springer Basel AG, Basel, 2012.

\bibitem{HST}
F. L. Hern\'andez, E. M. Semenov, and P. Tradacete,
Strictly singular operators on $L_p$ spaces and interpolation.
Proc. Amer. Math. Soc. \textbf{138} (2010), no. 2, 675--686.

\bibitem{J}
W. B. Johnson,
Operators into $L_p$ which factor through $\ell_p$.
J. London Math. Soc. (2) \textbf{14} (1976), no. 2, 333--339.

\bibitem{KP}
M. I. Kadec, A. Pelczynski,
Bases, lacunary sequences and complemented subspaces in the spaces $L_p$.
Studia Math. \textbf{21} (1961/62), 161--176.

\bibitem{Kato}
T. Kato,
Perturbation theory for nullity deficiency and other quantities of linear operators.
J. Analyse Math. \textbf{6} (1958), 273--322.

\bibitem{Krasnoselskii}
M. A. Krasnoselskii,
On a theorem of M. Riesz.
Dokl. Akad. Nauk SSSR \textbf{131} 246--248 (in Russian); translated as Soviet Math. Dokl. \textbf{1} (1960), 229--231.

\bibitem{KZPS}
M. A. Krasnoselskii, P. P. Zabrejko, E. I. Pustylnik, and P. E. Sobolevskij,
Integral operators in spaces of summable functions.
Monographs and Textbooks on Mechanics of Solids and Fluids. Leyden, The Netherlands: Noordhoff International Publishing XV (1976).

\bibitem{LNST}
J. Laitila, P. J. Nieminen, E. Saksman, H.-O. Tylli, Rigidity of composition operators on the Hardy space $H^p$. (Preprint, arXiv:1607.00113)


\bibitem{Le}
P. Lef\'evre,
When strict singularity of operators coincides with weak compactness.
J. Operator Theory \textbf{67} (2012), no. 2, 369--378.


\bibitem{LT1}
J. Lindenstrauss, L. Tzafriri.
Classical Banach Spaces. I. Sequence spaces.
Springer-Verlag, (1977).

\bibitem{LT2}
J. Lindenstrauss, L. Tzafriri,
Classical Banach spaces. II. Function spaces.
Springer-Verlag, (1979).


\bibitem{MN}
P. Meyer-Nieberg,
Banach lattices.
Universitext. Springer-Verlag, Berlin, 1991.

\bibitem{M}
V.~D.~Milman.
Operators of classes $C_0$ and $C_0^*$.
Functions theory, functional analysis and appl. \textbf{10}. (1970), 15--26 (in Russian).

\bibitem{Pel}
A. Pelczynski,
Banach spaces on which every unconditionally converging operator is weakly compact,
Bull. Acad. Polon. Sci. S\'er. Sci. Math. Astronom. Phys. \textbf{10} (1962), 641--648.

\bibitem{Pli:04}
A. Plichko,
Superstrictly singular and superstrictly cosingular operators.
Functional analysis and its applications, 239--255, North-Holland Math. Stud., 197, Elsevier, Amsterdam, 2004.

\bibitem{Read}
C. J. Read,
Strictly singular operators and the invariant subspace problem.
Studia Math. \textbf{132} (1999), no. 3, 203--226.

\bibitem{R}
S. D. Riemenschneider,
The $L$-characteristics of linear operators on $L^{1/\alpha}([0,1])$.
J. Funct. Anal. \textbf{8} (1971), 405--421.

\bibitem{Ru}
W. Rudin. Functional Analysis. McGraw-Hill, New York, 1973.

\bibitem{SS}
E. M. Semenov, F. A. Sukochev,
Banach--Saks index.
Sbornik: Mathematics \textbf{195} (2004), 263--285.

\bibitem{STH}
E. M. Semenov, P. Tradacete, and F. L. Hern\'andez,
Strictly singular operators on pairs of $L_p$ spaces.
Dokl. Math. \textbf{94} (2016), 450--452.

\bibitem{Weis:77}
L. Weis,
On perturbations of Fredholm operators in $L_p(\mu)$-spaces.
Proc. Amer. Math. Soc. \textbf{67} (1977), 287--292.

\bibitem{W}
R. J. Whitley,
Strictly singular operators and their conjugates.
Trans. Amer. Math. Soc. \textbf{113} (1964), 252--261.

\bibitem{ZK}
P. P. Zabrejko, M.A. Krasnoselskii,
On $L$-characteristics of operators.
Uspekhi Mat. Nauk,  Volume \textbf{19}, Issue 5 (119),(1964), 187-189  (in Russian).

\end{thebibliography}
\end{document}